\documentclass[11pt,letterpaper,reqno]{amsart}
\usepackage[T1]{fontenc}
\usepackage[utf8]{inputenc}
\usepackage[margin=1in]{geometry}
\usepackage{amssymb}
\usepackage{parskip}

\newtheorem{theorem}{Theorem}[section]
\newtheorem{lemma}[theorem]{Lemma}
\newtheorem{prop}[theorem]{Proposition}

\newcommand{\la}{\langle}
\newcommand{\ra}{\rangle}
\newcommand{\li}{\left}
\newcommand{\ri}{\right}
\newcommand{\R}{\mathbb{R}}

\newcommand{\C}{\mathbb{C}}

\newcommand{\N}{\mathbb{N}}
\newcommand{\re}{\operatorname{Re}}
\newcommand{\im}{\operatorname{Im}}

\usepackage{xcolor}

\definecolor{darkred}{rgb}{0.93,0.0,0.0}
\definecolor{mygreen}{rgb}{0.0,0.65,0.0}
\usepackage[
  colorlinks=true,
  linkcolor=darkred,     
  urlcolor=darkred,      
  citecolor=mygreen, 
  filecolor=darkred
]{hyperref}

\title[Non-Elliptic Quadratic $\mathcal{PT}$--Symmetric  Operators]{Non-Elliptic Quadratic $\mathcal{PT}$--Symmetric Operators and Similarity to Self-Adjoint Operators}
\author{Stepan Malkov}
\address{Department of Mathematics, University of California, Los Angeles, CA 90095, USA.}
\email{malkov@math.ucla.edu}
\date{\today}
\numberwithin{equation}{section}
\begin{document}
\begin{abstract}
We prove a similarity result for non-self-adjoint non-elliptic quadratic $\mathcal{PT}$--symmetric operators with real spectrum satisfying a dynamical averaging condition. In particular, we show that such an operator is weakly similar to a self-adjoint operator precisely when the fundamental matrix of the associated quadratic form is diagonalizable.
\end{abstract}
\maketitle
\tableofcontents
\section{Introduction and theorem statement}
$\mathcal{PT}$--symmetric operators, which are operators on $L^2(\R^n)$ left invariant under a successive application of a parity operator
\begin{equation}
\mathcal{P}u(x_1,x_2,\ldots,x_n) = u((-1)^{i_1}x_1,\ldots, (-1)^{i_n} x_n), \quad i_j \in \{0,1\}, \quad 1 \leq j \leq n, \quad (i_1,\ldots,i_n) \neq 0,
\end{equation}
and the time reversal operator
\begin{equation}
    \mathcal{T}u(x) = \overline{u(x)}, \quad x \in \R^n,
\end{equation}
play an important role in quantum mechanics and have become a recent and active area of research in mathematical physics \cite{Bender2024PT}. While it is typically assumed that a physically meaningful quantum observable is given by a self-adjoint operator on some Hilbert space $\mathcal{H},$ $\mathcal{PT}$--symmetric operators can exhibit real spectrum without necessarily being self-adjoint. Such operators are known as \textit{exact $\mathcal{PT}$--symmetric} operators, and one generally cannot expect them to be similar to self-adjoint operators. Instead, exact $\mathcal{PT}$--symmetric operators can be related to self-adjoint operators using a weaker notion of similarity which preserves the spectrum. 

For $j=1,2,$ let $\mathcal{H}_j$ be separable complex Hilbert spaces, and let $\mathcal{A}_j : \mathcal{D}(\mathcal{A}_j) \subseteq \mathcal{H}_j \to \mathcal{H}_j$ be closed, densely defined operators such that $\text{Spec}(\mathcal{A}_j) \subseteq \C$ is discrete and consists of isolated eigenvalues of finite algebraic multiplicity. For each $\lambda \in \text{Spec}(\mathcal{A}_j),$ let $E_{\lambda,j} = \ker((\mathcal{A}_j-\lambda I)^N)$ be the generalized eigenspace of $\mathcal{A}_j$ associated to the eigenvalue $\lambda,$ where $N \geq N(\lambda,j)$ is large enough. We assume that there exist linear subspaces $\mathcal{S}_j \subseteq \mathcal{D}(\mathcal{A}_j), j=1,2,$ such that $\mathcal{A}_j(\mathcal{S}_j) \subseteq \mathcal{S}_j,$ and suppose that $\bigcup_{\lambda \in \text{Spec}(\mathcal{A}_j)}E_{\lambda,j} \subseteq \mathcal{S}_j.$ 

We shall say that the operators $\mathcal{A}_1, \mathcal{A}_2$ are \textit{weakly similar} if there exists a linear bijection $S: \mathcal{S}_1 \to \mathcal{S}_2$ such that
\begin{equation}\label{eq:similar}
    S\mathcal{A}_1u = \mathcal{A}_2Su, \quad u \in \mathcal{S}_1.
\end{equation}
One may then directly show (see \cite[Proposition 1.1]{Caliceti2012QuadraticPT}) that weakly similar operators are isospectral. We will now state the assumptions and the main result proven in this paper. Throughout, we work on $\R^n$ and let $X=(x,\xi)$ be the phase space variables associated to the phase space $(x,\xi) \in T^*\R^n \cong \R^{2n}.$

Let $q=q(X)$ be a complex-valued quadratic form on $\R^{2n}$ such that $\re q$ is a positive semi-definite quadratic form on $\R^{2n},$
\begin{equation}
    \re q \geq 0.    
\end{equation}
For $T>0,$ we may define the time-symmetric average 
\begin{equation}\label{eq:avg_defn}
    \la \re q \ra_{\im q,T}(X) := \frac{1}{2T} \int_{-T}^T \re q(e^{tH_{\im q}}X)dt, \quad X \in \R^{2n},
\end{equation} 
where $H_b:= \partial_\xi b \cdot \partial_x-\partial_x b \cdot \partial_\xi$ is the Hamilton vector field associated to a symbol $b \in C^\infty(\R^{2n};\R).$
We will assume that the quadratic form $q$ satisfies the dynamical averaging condition 
\begin{equation}\label{ass_dyn_1}
    \la \re q\ra_{\im q,T}(X)>0, \quad 0 \neq X \in \R^{2n},
\end{equation}
for some (and hence for all) $T>0.$ 

\textit{Remark.} The assumption (\ref{ass_dyn_1}) is equivalent to the condition that the quadratic form $q$ satisfies
\begin{equation}\label{eq:sing_space}
    S:=\{X \in \R^{2n}: (H^k_{\im q} \re q)(X)=0,  \: \forall k \in \N_0\}=\{0\}.
\end{equation}
In the terminology of \cite{hitrik_starov_quadraticoperators}, \cite{hitrik_starov_i}, this is equivalent to the statement that the singular space $S \subseteq \R^{2n}$ associated to the quadratic form $q$ satisfies $S=\{0\}.$

Let $q^w(x,D)$ be the Weyl quantization of the quadratic symbol $q,$
\begin{equation}
    (q^w(x,D)u)(x) =(\text{Op}^w(q)u)(x):= \frac{1}{(2\pi)^n} \iint_{\R^{2n}} e^{i(x-y) \cdot \theta} q\li(\frac{x+y}{2},\theta\ri)u(y)dyd\theta,
\end{equation}
considered as a closed densely defined operator on $L^2(\R^n)$ with maximal domain
\begin{equation}
    \mathcal{D}(q^w(x,D)) = \{u \in L^2(\R^n): q^w(x,D)u \in L^2(\R^n)\}.
\end{equation}
Under the assumption (\ref{ass_dyn_1}), the spectrum of the quadratic operator $q^w(x,D)$ on $L^2(\R^n)$ is discrete and consists of isolated eigenvalues of finite algebraic multiplicity \cite{hitrik_starov_quadraticoperators}, and the generalized eigenfunctions of $q^w(x,D)$ lie in the Schwartz space $\mathcal{S}(\R^n)$ \cite[Section 2]{hitrik_starov_ii}. 

Following \cite{Caliceti2012QuadraticPT}, we will now introduce a $\mathcal{PT}$--symmetry assumption on the quadratic symbol $q.$ Let $\kappa: \R^n \to \R^n$ be a linear involution, that is, a linear map which satisfies
\begin{equation}
    \kappa^2 =I.
\end{equation}
Associated to the involution $\kappa$ is the parity operator $\mathcal{P},$ given by 
\begin{equation}
    (\mathcal{P}u)(x) = u(\kappa(x)). 
\end{equation}
We also define the time reversal operator $\mathcal{T}$ as
\begin{equation}
    (\mathcal{T}u)(x) = \overline{u(x)}.
\end{equation}
On the level of symbols, $\mathcal{P}$--symmetry for $q,$ 
\begin{equation}
    [q^w,\mathcal{P}]=0,
\end{equation}
is given by the condition
\begin{equation}
q(\kappa(x),\kappa^T(\xi))=q(x,\xi), \quad (x,\xi) \in \R^{2n},
\end{equation}
while $\mathcal{T}$--symmetry,
\begin{equation}
    [q^w,\mathcal{T}]=0,
\end{equation}
is given by
\begin{equation}
    \overline{q(x,-\xi)}=q(x,\xi), \quad (x,\xi) \in \R^{2n}.
\end{equation}
We require that the quadratic symbol $q$ is  $\mathcal{PT}$--symmetric with respect to $\kappa,$ thus satisfying
\begin{equation}\label{pt_ass}
    \overline{q \circ \mathcal{K}} = q \Longleftrightarrow \overline{q(\kappa(x),-\kappa^T(\xi))} = q(x,\xi), \quad (x,\xi) \in \R^{2n},
\end{equation}
for the map
\begin{equation}
    \mathcal{K}: \R^{2n} \to \R^{2n}, \quad \mathcal{K}(x,\xi) = (\kappa(x),-\kappa^T(\xi)). 
\end{equation}
Indeed, since
\begin{equation} (\mathcal{PT}) \circ \text{Op}^w(q) \circ (\mathcal{PT})^{-1}=\text{Op}^w(\overline{q \circ \mathcal{K}}),
\end{equation}
the assumption (\ref{pt_ass}) is equivalent to the statement that 
\begin{equation}\label{ass_pt_sym}
[q^w,\mathcal{PT}]=0.
\end{equation} 
The $\mathcal{PT}$--symmetry assumption (\ref{ass_pt_sym}) implies that $\text{Spec}(q^w(x,D))$ is symmetric with respect to the real axis.
Finally, we will assume that $q^w(x,D)$ has purely real spectrum on $L^2(\R^n),$ 
\begin{equation}\label{ass_spec_real}
    \text{Spec}(q^w(x,D)) \subseteq \R.
\end{equation}
We let 
\begin{equation}
    \sigma((x,\xi),(y,\eta)) := \xi \cdot y - \eta \cdot x, \quad (x,\xi)=X \in \R^{2n}, \quad (y,\eta)=Y \in \R^{2n},
\end{equation}
denote the canonical symplectic form on $\R^{2n}$ and its extension to the complex symplectic form on $\C^{2n}.$ We also let $q=q(X,Y): \C^{2n} \times \C^{2n} \to \C$ denote the polarization of $q,$ viewed as a symmetric bilinear form on $\C^{2n}.$
Following \cite[Section 21.5]{Hormander1985}, let $F: \C^{2n} \to \C^{2n}$ denote the Hamilton map associated to $q,$ defined via the identity 
\begin{equation}\label{eq:sigma_f_defn}
    q(X,Y) = \sigma(X,FY), \quad X,Y \in \C^{2n}.
\end{equation}
We note that $F$ is skew-symmetric with respect to $\sigma$ and is given in canonical coordinates by the \textit{fundamental matrix} 
\begin{equation}\label{eq:fund_matrix}
    F = \frac12\begin{bmatrix}
        \partial^2_{\xi x}q & \partial^2_{\xi\xi}q \\
        -\partial^2_{xx}q & -\partial^2_{x\xi}q
    \end{bmatrix},
\end{equation} 
where we use the notation $\partial^2_{\xi x}q=(\partial^2_{\xi_i x_j}q)_{1 \leq i,j \leq n}.$

The following theorem is the main result of this work, which generalizes the result of \cite{Caliceti2012QuadraticPT} to non-elliptic quadratic operators satisfying the averaging condition (\ref{ass_dyn_1}).

\begin{theorem}\label{thm_pt}
    Let $q: \R^{2n} \to \C$ be a complex-valued quadratic form that satisfies the averaging condition \textup{(\ref{ass_dyn_1})}, and let $\kappa: \R^n \to \R^n$ be a linear involution. Assume that the operator $q^w(x,D)$ satisfies the assumptions $\textup{(\ref{ass_pt_sym})}$ and $\textup{(\ref{ass_spec_real})}.$ Then, $q^w(x,D)$ is weakly similar to a self-adjoint operator in the sense of definition \textup{(\ref{eq:similar})} if and only if the fundamental matrix $F$ of $q$ is diagonalizable.
\end{theorem}
The plan of this paper is as follows. Following the ideas developed in \cite{Caliceti2012QuadraticPT}, we will prove Theorem \ref{thm_pt} in Section \ref{thm_pt_proof} by
reducing the stable and unstable Lagrangian manifolds associated to the Hamilton flow of $\frac{q}{i}$ to a normal form using the techniques of FBI transforms with quadratic phases, which will allow us to explicitly define the transformation $S$ satisfying the similarity condition (\ref{eq:similar}). Section \ref{example} discusses necessary and sufficient conditions for application of Theorem \ref{thm_pt} to a quadratic Kramers-Fokker-Planck operator. Finally, the appendix contains a proof of a density result for holomorphic polynomials in Bargmann spaces with strictly convex quadratic exponential weights, which is used in the proof of Theorem \ref{thm_pt}. 
 
\section{Similarity transformation and proof of Theorem \ref{thm_pt}}\label{thm_pt_proof}
The purpose of this section is to provide an explicit similarity transformation for $q^w(x,D)$ into a normal form and thereby complete the proof of Theorem \ref{thm_pt}. In particular, this approach is inspired by and closely follows the methods used in \cite{Caliceti2012QuadraticPT} and \cite{Viola2013}.

Let $q: \R^{2n} \to \C$ be a quadratic form that satisfies the assumptions (\ref{ass_dyn_1}), (\ref{ass_pt_sym}), (\ref{ass_spec_real}), and let $F$ denote the $2n \times 2n$ fundamental matrix of $q.$ For each eigenvalue $\lambda \in \text{Spec}(F),$ we define the associated generalized eigenspace 
\begin{equation}
  E_\lambda = \ker ((F-\lambda I)^{2n}) \subseteq \C^{2n}.  
\end{equation}
We now consider the unstable linear manifold 
\begin{equation}
    \Lambda^+ = \bigoplus_{\lambda \in \text{Spec}(F), \im \lambda>0} E_\lambda \subseteq \C^{2n}
\end{equation}
and the stable linear manifold 
\begin{equation}
    \Lambda^- = \bigoplus_{\lambda \in \text{Spec}(F), \im \lambda<0} E_\lambda \subseteq \C^{2n}
\end{equation}
associated to the Hamilton flow of $\frac{q}{i}.$ We note that $\lambda \in \text{Spec}(F)$ if and only if $-\lambda \in \text{Spec}(F)$ \cite[Section 21.5]{Hormander1985}, with agreement of algebraic multiplicities. Moreover, if $q$ satisfies (\ref{ass_dyn_1}), it follows from \cite[Section 6]{malkov2026harmonic} that there exists an I-Lagrangian, R-symplectic linear manifold $\Lambda \subseteq \C^{2n}$ such that $\re q|_\Lambda>0$ in the sense of quadratic forms. In view of (\ref{eq:sigma_f_defn}), this implies that $\text{Spec}(F) \cap \R = \varnothing,$ so we conclude that there are exactly $n$ eigenvalues of $F$ in the upper and lower half-plane (counting multiplicity), respectively.  

The arguments of \cite[Proposition 3.3]{Sjostrand1974} show that the manifolds $\Lambda^+$ and $\Lambda^-$ are complex Lagrangian planes, and since $F(\Lambda^-) \subseteq \Lambda^-, F(\Lambda^+) \subseteq \Lambda^+,$ we see that 
\begin{equation}
    q(X,Y) = \sigma(X,FY) = 0, \quad X,Y \in \Lambda^- 
\end{equation}
and
\begin{equation}
    q(X,Y) = \sigma(X,FY) = 0, \quad X,Y \in \Lambda^+.
\end{equation}
We now use the averaging assumption (\ref{ass_dyn_1})  to conclude the sign definiteness of $\Lambda^+$ and $\Lambda^-.$ In particular, it is shown in  \cite[Proposition 2.1]{Viola2013} that $\Lambda^+$ is a strictly positive Lagrangian plane, in the sense that
\begin{equation}\label{eq:lambda_pos}
    \frac{1}{i} \sigma(X,\overline{X})>0, \quad X \neq 0, \quad X \in \Lambda^+, \ 
\end{equation}
and that $\Lambda^-$ is a strictly negative Lagrangian plane, in the sense that
\begin{equation}\label{eq:lambda_neg}
    \frac{1}{i} \sigma(X,\overline{X})<0, \quad X \neq 0, \quad X \in \Lambda^-, \ 
\end{equation}
(see \cite[Proposition 3.3]{Sjostrand1974}, and also \cite[Appendix A]{Caliceti2012QuadraticPT} for a general discussion of positivity). 

We now describe the normal form reduction via an associated FBI transform. A holomorphic quadratic form $\varphi: \C^n_x \times \C^n_y \to \C$ satisfying 
\begin{equation}\label{eq:phase_adm}
    \im \partial^2_{yy} \varphi >0, \quad \det \partial_{xy}^2 \varphi \neq 0
\end{equation}
gives rise to the Fourier-Bros-Iagolnitzer (FBI) transform
\begin{equation}
    T_\varphi u(x) = c_{\varphi,n} \int_{\R^n} e^{i\varphi(x,y)}u(y)dy, \quad x \in \C^{n},
\end{equation}
where $c_{\varphi,n}>0$ is a constant such that $T_\varphi$ is unitary as a map
\begin{equation}
    T_\varphi: L^2(\R^{n}) \to H_\Phi(\C^n).
\end{equation}
Here, the Bargmann space $H_\Phi(\C^n)$ is defined as 
\begin{equation}
    H_\Phi(\C^n) :=H(\C^n) \cap L^2_\Phi(\C^n),
\end{equation} 
where $L^2_\Phi(\C^n)$ is the exponentially weighted $L^2$--space
\begin{equation}
    L^2_{\Phi}(\C^n):= L^2(\C^n,e^{-2\Phi}d\mu), \quad \mu := \text{Lebesgue measure on }\C^n,
\end{equation}
and 
\begin{equation}\label{eq:phi_defn}
    \Phi(x):=\sup_{y \in \R^n} (-\im \varphi(x,y))
\end{equation}
is strictly plurisubharmonic (for an overview of FBI transform tools, see \cite[Section 4]{malkov2026harmonic}). Associated to $T_\varphi$ is a complex-linear canonical transformation $\kappa_{T_\varphi},$ implicitly given by
\begin{equation}\label{eq:canonical_defn}
    \kappa_{T_\varphi}:(y,-\partial_y\varphi(x,y)) \to (x,\partial_x \varphi(x,y)), \quad x,y \in \C^n, 
\end{equation}
and an I-Lagrangian, R-symplectic linear subspace $\Lambda_{\Phi}$ defined by
\begin{equation}\label{eq:lambda_phi_defn}
\Lambda_{\Phi}:=\kappa_{T_\varphi}(\R^{2n})=\li\{\li(x,\frac{2}{i}\partial_x \Phi(x)\ri): x \in \C^n\ri\} \subseteq \C^{2n}.
\end{equation}
For a quadratic form $q$ satisfying the averaging condition (\ref{ass_dyn_1}), we then define the Weyl quantization $\widetilde{q}^w(x,D)$ of the holomorphic quadratic form $\widetilde{q} = q \circ \kappa_{T_\varphi}^{-1}$ on $\C^{2n}$ as an unbounded operator acting on $H_\Phi(\C^n)$ and given by 
\begin{equation}
    (\widetilde{q}^w(x,D)u)(x):= \frac{1}{(2\pi )^n} \iint_{\Gamma(x)} e^{i (x-y) \cdot \theta} \widetilde{q}\li(\frac{x+y}{2},\theta\ri) u(y)dy \wedge d\theta,
\end{equation}
with
\begin{equation}
    \Gamma(x):= \li\{\li(y,\frac{2}{i}\partial_x \Phi\li(\frac{x+y}{2}\ri)+ic (\overline{x-y})\ri): y \in \C^n\ri\}, \quad c>0.
\end{equation}
When equipped with its maximal domain, $\widetilde{q}^w(x,D)$ becomes a closed densely defined operator on $H_\Phi(\C^n)$ with discrete spectrum (see \cite{hitrik_starov_quadraticoperators}). The FBI transform satisfies an exact version of Egorov's theorem, yielding
\begin{equation}\label{eq:egorov}
    T_\varphi \circ q^w = \widetilde{q}^w \circ T_\varphi.
\end{equation}
As a consequence of (\ref{eq:lambda_pos}) and (\ref{eq:lambda_neg}), we may now apply the example in \cite[Section 1.2]{almost_holomorphic_ref} to conclude that there exists a holomorphic quadratic form $\varphi,$ satisfying (\ref{eq:phase_adm}), and a quadratic exponential weight $\Phi_0: \C^n \to \R$ such that $\varphi$ is associated to a unitary FBI transform $T: L^2(\R^n) \to H_{\Phi_0}(\C^n)$ and a canonical linear transformation $\kappa_T$ which satisfies
\begin{equation}
    \kappa_T(\Lambda^-)= \{(x,\xi) \in \C^{2n}: x=0\}
\end{equation}
and
\begin{equation}
    \kappa_T(\Lambda^+) = \{(x,\xi) \in \C^{2n}: \xi =0\}.
\end{equation}
 The strict positivity of $\Lambda^+$ and an application of \cite[Proposition A.9]{Caliceti2012QuadraticPT} further shows that
\begin{equation}\label{eq:convex_0}
    \Phi_0(x) \sim |x|^2, \quad x \in \C^n.
\end{equation}
Thus, the quadratic form $\widetilde{q}=q \circ \kappa_T^{-1}$ vanishes on $\kappa_T(\Lambda^+)$ and $\kappa_T(\Lambda^-),$ so we conclude that $\widetilde{q}(x,\xi) =Rx \cdot \xi$ for some complex $n \times n$ matrix $R.$ Moreover, we note that the fundamental matrix $\widetilde{F}$ of $\widetilde{q}$ satisfies
\begin{equation}
    \widetilde{F} = \kappa_T \circ F \circ \kappa_T^{-1} = \frac12 \begin{bmatrix}
        R & 0 \\
        0 & -R^T\\
    \end{bmatrix},
\end{equation}
and hence $R:\C^n \to \C^n$ is similar to $2F: \Lambda^+ \to \Lambda^+.$

Lastly, we bring $R$ into Jordan normal form. If $M$ is the Jordan normal form of $R$ and $R=PMP^{-1}$ its Jordan decomposition, then the canonical linear transformation
\begin{equation}
    \C^{2n} \ni (x,\xi) \to (P^{-1}x,P^T\xi) \in \C^{2n}
\end{equation}
is associated to the unitary operator $U: H_{\Phi_0}(\C^n) \to H_{\Phi_1}(\C^n)$ given by 
\begin{equation}
Uu(x) = |\det P| u(Px),    
\end{equation}
where $\Phi_1$ is the rescaled quadratic weight 
\begin{equation}
    \Phi_1(x)=\Phi_0(Px).
\end{equation}
In view of (\ref{eq:convex_0}), we also note that 
\begin{equation}
    \Phi_1(x) \sim |x|^2, \quad x \in \C^n.
\end{equation}
Following a unitary conjugation by $U$, one obtains an operator whose Weyl symbol is of the form $\widetilde{q}(x,\xi) = Mx \cdot \xi,$ where $M$ is a matrix in Jordan normal form. It follows from (\ref{eq:fund_matrix}) that the fundamental matrix of this quadratic form is given by 
\begin{equation}
    \widetilde{F} = \frac12\begin{bmatrix}
        M & 0 \\
        0 & -M^T \\
    \end{bmatrix},
\end{equation}
so in particular we conclude that
\begin{equation}
    \text{Spec}(M) =\text{Spec}(2F) \cap \{\im \lambda>0\}.
\end{equation}

The above discussion may be summarized in the following result (see also \cite[Proposition 2.2]{Viola2013}).
\begin{prop}\label{prop_unitary}
    Let $q: \R^{2n} \to \C$ be a quadratic form satisfying the averaging condition \textup{(\ref{ass_dyn_1})}, and let $F$ denote its fundamental matrix. Then, there exists a strictly convex quadratic weight function $\Phi_1: \C^n \to \R$ and a unitary FBI transform $T: L^2(\R^n) \to H_{\Phi_1}(\C^n)$ such that the operator \begin{equation}
        q^w(x,D): L^2(\R^n) \to L^2(\R^n),
    \end{equation}
    equipped with its maximal domain 
    \begin{equation}
        \mathcal{D}(q^w(x,D)) =\{u \in L^2(\R^n): q^w(x,D)u \in L^2(\R^n)\},
    \end{equation}
    is unitarily equivalent via $T$ to a quadratic operator
    \begin{equation}
    \widetilde{q}^w(x,D): H_{\Phi_1}(\C^n) \to H_{\Phi_1}(\C^n),
    \end{equation}
    equipped with maximal domain 
    \begin{equation}
         \mathcal{D}(\widetilde{q}^w) =\{u \in H_{\Phi_1}(\C^n): \widetilde{q}^w(x,D)u \in H_{\Phi_1}(\C^n)\}.
    \end{equation}
    Here, $\widetilde{q}(x,\xi) = Mx \cdot \xi,$ where the complex $n \times n$ matrix $M$ is in Jordan normal form and the eigenvalues of $M,$ counting algebraic multiplicity, are precisely the eigenvalues of $2F$ in the upper half-plane. 
\end{prop}
Letting $\lambda_1,\ldots, \lambda_n$ denote the eigenvalues of $F$ in the upper half-plane, Proposition \ref{prop_unitary} implies that the action of $\widetilde{q}^w(x,D)$ is given by
\begin{equation}\label{eq:q_fbi_reduced}
    \widetilde{q}^w(x,D)u = \frac1i\sum_{j=1}^n \lambda_ju +2\sum_{j=1}^n \lambda_j x_jD_{x_j}u+\sum_{j=1}^{n-1} \gamma_jx_{j+1}D_{x_j}u, \quad \gamma_j \in \{0,1\}, \quad u \in \mathcal{D}(\widetilde{q}^w(x,D)).
\end{equation}
We now argue that for the standard exponential weight $\Phi(x) = \frac12 |x|^2,$ the operator $\widetilde{q}^w(x,D)$ is a densely defined unbounded operator on $H_\Phi(\C^n).$ Indeed, we claim that the space 
\begin{equation}\label{eq:hol_poly}
    \mathcal{S}_1 = \text{span}\{x^\alpha: x \in \C^n,\alpha \in \N_0^n\}
\end{equation}
of holomorphic polynomials is a dense subspace of both $H_{\Phi}(\C^n)$ and $H_{\Phi_1}(\C^n).$ Both conclusions follow as a direct consequence of Lemma \ref{lemma_dense}, stated and proved in the appendix.

Now, if $q^w(x,D)$ satisfies (\ref{ass_pt_sym}), it follows from (\ref{eq:real_im_spec}) that the assumption (\ref{ass_spec_real}) is equivalent to the statement that $\lambda_1,\ldots, \lambda_n \in i \R.$ If $F$ is diagonalizable, then the similarity of $F$ and $\widetilde{F}$ implies that $M$ must be diagonal, and hence $\gamma_1 = \ldots = \gamma_{n-1}=0.$ Then, the 
symbol of the quadratic operator
\begin{equation}\label{eq:q_fbi_self_adj}
   \widetilde{q}^w(x,D):H_\Phi(\C^n) \to H_\Phi(\C^n), \quad  \widetilde{q}^w(x,D) = \frac1i \sum_{j=1}^n \lambda_j +2 \sum_{j=1}^n \lambda_j x_jD_{x_j},
\end{equation}
satisfies 
\begin{equation}
    \widetilde{q}\li(x, \frac2i \partial_x \Phi(x)\ri) = 2\sum_{j=1}^n \frac{\lambda_j}{i} |x_j|^2,
\end{equation}
and is hence real-valued and elliptic on $\Lambda_\Phi.$ Thus, the operator (\ref{eq:q_fbi_self_adj}) has discrete spectrum and is self-adjoint when equipped with its maximal domain. Moreover, the strict convexity of $\Phi_1$ and the explicit description of the basis of eigenfunctions of $\widetilde{q}^w(x,D)$ in \cite[Lemma 4.3]{Viola2013} together show that the eigenfunctions of (\ref{eq:q_fbi_self_adj}) on $H_\Phi(\C^n)$ and $H_{\Phi_1}(\C^n)$ are contained in the space (\ref{eq:hol_poly}) of holomorphic polynomials. We therefore see that $\widetilde{q}^w(x,D): H_{\Phi_1}(\C^n) \to H_{\Phi_1}(\C^n)$ is weakly similar to the self-adjoint operator $\widetilde{q}^w(x,D): H_\Phi(\C^n) \to H_\Phi(\C^n)$ in the sense of definition (\ref{eq:similar}) using the similarity operator $S=I$ on the dense subspace (\ref{eq:hol_poly}) of holomorphic polynomials. 

Conversely, if $q^w(x,D)$ is weakly similar to a self-adjoint operator $A$ via some linear bijection $S,$ then for any $\lambda \in \text{Spec}(q^w(x,D)),$ any $N \geq 1,$ and any generalized eigenfunction $u \in \ker((q^w-\lambda)^N)$ of $q^w(x,D),$ one has
\begin{equation}
    S(q^w(x,D)-\lambda)^Nu= (A-\lambda)^N Su =0 \implies (A-\lambda)Su=0,
\end{equation}
which implies that $u$ is an eigenfunction of $q^w(x,D)$ with eigenvalue $\lambda.$ Fix $\alpha = (\alpha_1, \ldots,\alpha_n) \in \N_0^n,$ set $\lambda_\alpha = \sum_{j=1}^n \frac{\lambda_j}{i} (1+2\alpha_j),$ and note that $\lambda_\alpha \in \text{Spec}(q^w(x,D)) = \text{Spec}(\widetilde{q}^w(x,D))$ in view of (\ref{eq:eig_formula_kfp}). Then, (\ref{eq:q_fbi_reduced}) shows that $x^\alpha$ is a generalized eigenfunction of $\widetilde{q}^w(x,D)$ on $H_{\Phi_1}(\C^n)$ with eigenvalue $\lambda_\alpha,$ and hence also an eigenfunction of $\widetilde{q}^w(x,D)$ with the same eigenvalue. Since $\alpha \in \N_0^n$ is arbitrary, this shows that
\begin{equation}
    \sum_{j=1}^{n-1} \gamma_j x_{j+1} D_{x_j} x^\alpha =0, \quad \forall \alpha \in \N_0^n, 
\end{equation}
which yields that $\gamma_1 = \ldots = \gamma_{n-1}=0,$ and since $F$ and $\widetilde{F}$ are similar, lets us conclude that $F$ is diagonalizable. This concludes the proof of Theorem \ref{thm_pt}.

\section{Example: quadratic Kramers-Fokker-Planck operator}\label{example}
In this section, we provide a necessary and sufficient condition for a quadratic Kramers-Fokker-Planck operator to satisfy the statement of Theorem \ref{thm_pt}. 

Let $V=V(x): \R^n \to \R$ be a real quadratic form. Extend $x \in \R^n$ to the coordinates $X=(x,y) \in \R^{2n}$ with the corresponding dual coordinates $\Xi=(\xi,\eta) \in \R^{2n},$ and fix the canonical coordinates $Y=(X,\Xi)=(x,y,\xi,\eta) \in \R^{4n}$ associated to the phase space $T^* \R^{2n}\cong \R^{4n}.$
For fixed $b>0,$ define the complex-valued quadratic form $q: \R^{4n} \to \C$ as
\begin{equation}\label{eq:quad_kfp}
    q(x,y,\xi,\eta) = \frac12 (|y|^2+|\eta|^2) + ib(y\cdot \xi - V'(x) \cdot \eta).
\end{equation}
Then, the Weyl quantization of $q$ is given by the quadratic Kramers-Fokker-Planck operator
\begin{equation}\label{eq:kfp_op}
    q^w(X,D_X)= \frac12 (|y|^2+D_y^2)+b(y \cdot \partial_x-V'(x) \cdot \partial_y).
\end{equation}
The operator $q^w(X,D_X)$ is a closed densely defined operator on $L^2(\R^{2n})$ with maximal domain
    \begin{equation}
        \mathcal{D}(q^w(X,D_X)) = \{u \in L^2(\R^{2n}): q^w(X,D_X)u \in L^2(\R^{2n})\}.
    \end{equation}
Let $A=V''$ denote the Hessian of $V,$ and let $A=P\Lambda P^{-1}$ be its orthogonal diagonalization, where the $n \times n$ matrix $\Lambda=\text{diag}(\mu_1,\ldots,\mu_n)$ is diagonal. The symplectic change of coordinates $\widetilde{\kappa}: \R^{4n} \to \R^{4n},$ given by
\begin{equation}
    \R^{4n} \ni (x,y,\xi,\eta) \to (Px,Py,P\xi,P\eta) \in \R^{4n},
\end{equation}
then corresponds to a unitary operator 
\begin{equation}
    U: L^2(\R^{2n}) \to L^2(\R^{2n}), \quad Uu(x,y) := u(Px,Py),
\end{equation}
such that the conjugated operator $U q^w(X,D_X)U^{-1}$ is the Weyl quantization of the complex-valued quadratic form $\widetilde{q} = q \circ \widetilde{\kappa},$ given by
\begin{equation}\label{eq:kfp_v_diag}
     \widetilde{q}(x,y,\xi,\eta) := \frac12 (|y|^2+|\eta|^2) + ib(y\cdot \xi - \Lambda x \cdot \eta).
\end{equation}
Then, if the averaging assumption (\ref{ass_dyn_1}) holds for $\widetilde{q},$ it also holds for $q.$ Moreover, the operator $\widetilde{q}^w(X,D_X),$ equipped with its maximal domain, is a closed densely defined operator on $L^2(\R^{2n})$ which is isospectral to $q^w(X,D_X),$ so if (\ref{ass_spec_real}) holds for $\widetilde{q}^w(x,D),$ it also holds for $q^w(x,D).$
We define the involution 
\begin{equation}
    \kappa: \R^{2n} \to \R^{2n}, \quad \kappa(x,y) = (-x,-y),
\end{equation}
and note that 
\begin{equation}
    \widetilde{q}(x,y,\xi,\eta) = \overline{\widetilde{q}(-x,-y,\xi,\eta)}.
\end{equation}
Since the associated parity operator $\mathcal{P}$ commutes with $U,$ if the $\mathcal{PT}$--symmetry assumption (\ref{ass_pt_sym}) with respect to $\kappa$ holds for $\widetilde{q}^w(x,D),$ then it also holds for $q^w(x,D).$ Hence, we may without loss of generality replace $q$ with $\widetilde{q}$ in the following discussion.

We claim that the averaging condition (\ref{ass_dyn_1}) holds under the assumption that the diagonal matrix $\Lambda$ in (\ref{eq:kfp_v_diag}) is invertible. We may compute
\begin{equation}
    \re q(x,y,\xi,\eta)=\frac12(|y|^2+|\eta|^2), \quad H_{\im q}=b(y \cdot \partial_x -\Lambda x \cdot \partial_y+\Lambda \eta \cdot \partial_\xi -\xi \cdot \partial_\eta), 
\end{equation}
\begin{equation}
    H_{\im q} \re q(x,y,\xi,\eta) = -b(\xi \cdot \eta +\Lambda x \cdot y),
\end{equation}
and
\begin{equation}
    H^2_{\im q} \re q(x,y,\xi,\eta)=b^2(|\Lambda x|^2+ |\xi|^2-\Lambda y \cdot y -\Lambda \eta \cdot \eta).
\end{equation}
Hence, invertibility of $\Lambda$ implies that the singular space $S$ of $q$ is
\begin{equation}\label{eq:sing_empty}
    S=\{Y \in \R^{4n}: H^k_{\im q} \re q(Y)=0, \: \forall k \in \N_0\}=\{0\},
\end{equation}
which yields the averaging condition (\ref{ass_dyn_1}) in view of the remark (\ref{eq:sing_space}) (see \cite{hitrik_starov_quadraticoperators}, \cite[Section 1]{hitrik_starov_i}).

We now argue that if $0 \neq b^2 \mu_j \leq \frac14$ for $j=1,2,\ldots,n,$ then the spectrum of the quadratic operator (\ref{eq:kfp_op}) consists of real eigenvalues of finite algebraic multiplicity. Indeed, recall that if $q$ is a complex-valued quadratic form satisfying the averaging condition (\ref{ass_dyn_1}), then the spectrum of $q^w(x,D)$ is discrete and consists of eigenvalues of finite algebraic multiplicity, taking the form
\begin{equation}\label{eq:eig_formula_kfp}
\text{Spec}(q^w(x,D)) =\li\{\sum_{\lambda \in K \cap \text{Spec}(F)}\frac{\lambda}{i}(r_\lambda+2k_\lambda):  k_\lambda \in \N_0 \ri\}, \quad K = \{z \in \C: \im z >0\},
\end{equation}
where $F$ is the fundamental matrix (\ref{eq:fund_matrix}) of the quadratic form (\ref{eq:kfp_v_diag}) and $r_\lambda \geq 1$ is the algebraic multiplicity associated to the eigenvalue $\lambda$ of $F$ \cite[Theorem 1.2.2]{hitrik_starov_quadraticoperators}. In fact, if $q^w(x,D)$ satisfies the $\mathcal{PT}$--symmetry condition (\ref{ass_pt_sym}), \cite[Proposition 2.1]{Caliceti2012QuadraticPT} and (\ref{eq:eig_formula_kfp}) together show that
\begin{equation}\label{eq:real_im_spec}
    \text{Spec}(q^w(x,D)) \subseteq \R \Longleftrightarrow \text{Spec}(F) \subseteq i\R.
\end{equation}
The $4n \times 4n$ fundamental matrix $F$ associated to (\ref{eq:kfp_v_diag}) in canonical coordinates $(X,\Xi)$ takes the form
\begin{equation}\label{eq:kfp_fund}
    F = \frac12 \begin{bmatrix}
        0 & ibI_n& 0 & 0\\
        -ib\Lambda& 0 & 0 & I_n\\
        0& 0 &  0 & ib\Lambda\\
        0 & -I_n & -ibI_n & 0 \\
    \end{bmatrix}.
\end{equation}
Note that each block in $F$ is a diagonal matrix, and hence the characteristic polynomial of $F$ is given by 
\begin{equation}
    \chi_F(\lambda) = \prod_{j=1}^n \chi_{F_j}(\lambda), \quad F_j = \frac12 \begin{bmatrix}
        0 & ib & 0 & 0 \\
        -ib\mu_j & 0 & 0 & 1 \\
        0 & 0 & 0 & ib\mu_j \\
        0 & -1 & -ib & 0 \\
    \end{bmatrix}.
\end{equation}
Straightforward computations then show that 
\begin{equation}
    \chi_F(\lambda) = \prod_{j=1}^n \li(\li(\lambda^2-\frac{b^2\mu_j}{4}\ri)^2+\frac{\lambda^2}{4}\ri),
\end{equation}
which has roots 
\begin{equation}
    \lambda= \frac14 (\pm i \pm \sqrt{4b^2 \mu_j-1}).
\end{equation}
    If we now assume that $0 \neq b^2\mu_j \leq \frac14$ for $j=1,2,\ldots,n,$ then each eigenvalue of $F$ is purely imaginary, and hence (\ref{eq:eig_formula_kfp}) implies that the spectrum of $q^w(x,D)$ is purely real, thus satisfying the assumption (\ref{ass_spec_real}). If $0 \neq b^2 \mu_j < \frac14$ for $j=1,2,\ldots,n,$ each $F_j$ has four distinct eigenvalues and $F$ is hence diagonalizable. Moreover, if $b^2\mu_j=\frac14$ for some $j,$ then it is clear that the first three columns of
    \begin{equation}
        F_j+\frac{i}{4} = \frac12\begin{bmatrix}
        \frac{i}{2} & ib & 0 & 0 \\
        -ib\mu_j & \frac{i}{2} & 0 & 1 \\
        0 & 0 & \frac{i}{2} & ib\mu_j \\
        0 & -1 & -ib & \frac{i}{2} \\
    \end{bmatrix}
    \end{equation}
    are linearly independent, showing that $\dim \ker (F_j+\frac{i}{4}) = 1$ and hence demonstrating that $F$ is not diagonalizable. Thus, for $0 \neq b^2 \mu_j \leq \frac14, j=1,2,\ldots,n,$ the assumptions (\ref{ass_dyn_1}), (\ref{ass_pt_sym}), and (\ref{ass_spec_real}) of Theorem \ref{thm_pt} hold for the quadratic Kramers-Fokker-Planck operator (\ref{eq:kfp_op}), 
    so we conclude that the operator $q^w(x,D)$ is weakly similar to a self-adjoint operator in the sense of definition (\ref{eq:similar}) if and only if $b^2\mu_j \neq \frac14$ for all $j=1,2,\ldots,n.$

\appendix
\section{Density of holomorphic polynomials in Bargmann spaces}\label{appendix_a}
The purpose of this appendix is to provide a proof of the following lemma, originally due to \cite[Section 2]{Viola2013}. Our proof is similar, but is somewhat more direct.
\begin{lemma}\label{lemma_dense}
    Let $\Phi: \C^n \to \R$ be a strictly convex quadratic weight. Then, the space \textup{(\ref{eq:hol_poly})} of holomorphic polynomials is dense in $H_{\Phi}(\C^n).$
\end{lemma}
\begin{proof}
    It is clear that $\mathcal{S}_1 \subseteq H_\Phi(\C^n).$ Given a strictly convex quadratic weight $\Phi,$ we note that there exists a holomorphic quadratic form $\varphi: \C^n_x \times \C^n_y \to \C$ satisfying \begin{equation}
    \im \partial^2_{yy} \varphi >0, \quad \det \partial_{xy}^2 \varphi \neq 0,
    \end{equation}
    such that the canonical transformation $\kappa$ implicitly defined via
    \begin{equation}
    \kappa:(y,-\partial_y\varphi(x,y)) \to (x,\partial_x \varphi(x,y)), \quad x,y \in \C^n, 
    \end{equation}
    satisfies $\kappa(\R^{2n}) = \Lambda_\Phi$ (see \cite[Section 13.3]{Zworski2012}). The strict convexity of $\Phi$ implies that the Lagrangian manifold $\Lambda=\{(x,\xi) \in \C^{2n}: \xi =0 \}$ is strictly positive with respect to $\Lambda_\Phi$ (see \cite[Proposition A.9]{Caliceti2012QuadraticPT}). Thus, $\kappa^{-1}(\Lambda)$ is strictly positive with respect to $\kappa^{-1}(\Lambda_\Phi)=\R^{2n}$ and therefore is of the form \begin{equation}
        \kappa^{-1}(\Lambda) = \{(y,\eta) \in \C^{2n}: \eta = Ay\}
    \end{equation}
    for some complex $n \times n$ symmetric matrix $A$ with $\im A>0$ \cite[Example A.6]{Caliceti2012QuadraticPT}. Letting $T: L^2(\R^n) \to H_\Phi(\C^n)$ denote the FBI transform associated to $\varphi,$ quadratic stationary phase \cite[Lemma 13.2]{Zworski2012} implies that
    \begin{equation}
        T(e^{\frac{iAy \cdot y}{2}}) = Ce^{ig(x)}, \quad C \neq 0,
    \end{equation}
    for the holomorphic quadratic form 
    \begin{equation}
    g(x) = \text{vc}_{y \in \C^{n}} \li(\varphi(x,y) + \frac{Ay \cdot y}{2}\ri), \quad x \in \C^n    
    \end{equation}
    (here ``vc'' stands for critical value), and \begin{equation}
        \kappa(\{(y,Ay):y \in \C^n\}) = \{(x,g'(x)): x \in \C^n\}=\{(x,\xi) \in \C^{2n}: \xi =0\},
    \end{equation} 
    which implies that $g=0$ and lets us conclude that
    \begin{equation}\label{eq:ground_state}
        T(e^{\frac{iAy \cdot y}{2}}) = C, \quad C \neq 0.
    \end{equation}
    We then recall that \begin{equation}
        \mathcal{S}=\text{span}\{y^\alpha e^{\frac{iAy\cdot y}{2}}:  \alpha \in \N_0^n \}
    \end{equation}
    is dense in $L^2(\R^n)$ (see \cite[Lemma 3.12]{Sjostrand1974}). Indeed, suppose $f \in L^2(\R^n)$ is such that 
    \begin{equation}
        \int_{\R^n} y^\alpha e^{\frac{i Ay \cdot y}{2}}\overline{f(y)}dy = 0, \quad \forall \alpha \in \N_0^n.
    \end{equation}
    Letting $\mathcal{F}$ denote the Fourier transform and noting that $\im A>0,$ one then has that $\mathcal{F}(\overline{f} e^{\frac{iAy \cdot y}{2}})$ is entire on $\C^n$ and 
    \begin{equation}
        \mathcal{F}(\overline{f}y^\alpha e^{\frac{iAy \cdot y}{2}})(0) = (-1)^\alpha D^\alpha|_{\xi=0} \mathcal{F}(\overline{f} e^{\frac{iAy \cdot y}{2}})=0, \quad \forall \alpha \in \N_0^n,
    \end{equation}
    so $\mathcal{F}(\overline{f} e^{\frac{iAy \cdot y}{2}})=0$ and hence $f=0.$
    Finally, recalling the space $\mathcal{S}_1$ of holomorphic polynomials in (\ref{eq:hol_poly}), by Egorov's theorem, one has that $T(\mathcal{S}) \subseteq \mathcal{S}_1,$ so by the unitarity of $T,$ we conclude that the holomorphic polynomials are dense in $H_\Phi(\C^n).$
\end{proof}
\bibliographystyle{alpha}
\bibliography{refs_pt}
\end{document}